\documentclass[a4paper,oneside]{amsart}

\usepackage{geometry}
\usepackage[english]{babel}
\usepackage{amsmath}
\usepackage[utf8]{inputenc}
\usepackage[T1]{fontenc}
 
\usepackage{amsthm}
\usepackage{amssymb}
\usepackage[all]{xy}
\usepackage{MnSymbol}
\usepackage[,colorlinks=true,citecolor= DarkBlue,linkcolor=Black]{hyperref}
\usepackage[svgnames,table]{xcolor}

\theoremstyle{definition}
\newtheorem{dfn}{Definition}[section]

\theoremstyle{plain}
\newtheorem{thm}{Theorem}
\newtheorem*{thm*}{Theorem}
\newtheorem{prop}[dfn]{Proposition}
\newtheorem*{prop*}{Proposition}
\newtheorem{lem}[dfn]{Lemma}
\newtheorem*{lem*}{Lemma}
\newtheorem{cor}[dfn]{Corollary}

\newtheorem{fact}{Fact}
\newtheorem*{fact*}{Fact}

\theoremstyle{remark}
\newtheorem{rem}[dfn]{Remark}

\newcommand{\C}{\mathbf{C}}
\newcommand{\R}{\mathbf{R}}

\newcommand{\Z}{\mathbf{Z}}

\newcommand{\gl}{\mathfrak{gl}}

\newcommand{\su}{\mathfrak{su}}
\renewcommand{\sp}{\mathfrak{sp}}
\newcommand{\heis}{\mathfrak{heis}}

\newcommand{\co}{\mathfrak{co}}
\renewcommand{\a}{\mathfrak{a}}
\newcommand{\g}{\mathfrak{g}}
\newcommand{\h}{\mathfrak{h}}    
\newcommand{\m}{\mathfrak{m}}
\renewcommand{\o}{\mathfrak{so}}
\newcommand{\p}{\mathfrak{p}}
\newcommand{\s}{\mathfrak{s}}
\renewcommand{\u}{\mathfrak{u}}
\newcommand{\z}{\mathfrak{z}}

\newcommand{\e}{\mathrm{e}}

\DeclareMathOperator{\Diff}{Diff}
\DeclareMathOperator{\Hom}{Hom}
\DeclareMathOperator{\Isom}{Isom}
\DeclareMathOperator{\Ad}{Ad}  
\DeclareMathOperator{\ad}{ad}
\DeclareMathOperator{\Stab}{Stab}
\DeclareMathOperator{\Ker}{Ker}

\DeclareMathOperator{\Span}{Span}
\DeclareMathOperator{\Aut}{Aut}
\DeclareMathOperator{\Kill}{Kill}

\DeclareMathOperator{\GL}{GL}

\DeclareMathOperator{\PSL}{PSL}
\DeclareMathOperator{\PO}{PO}

\DeclareMathOperator{\SO}{SO}
\DeclareMathOperator{\SU}{SU}
\DeclareMathOperator{\PSO}{PSO}
\DeclareMathOperator{\PSU}{PSU}
\DeclareMathOperator{\Aff}{Aff}
\DeclareMathOperator{\Sp}{Sp}
\DeclareMathOperator{\Spin}{Spin}

\DeclareMathOperator{\Rk}{Rk}
\DeclareMathOperator{\Conf}{Conf}
\DeclareMathOperator{\Tr}{Tr}
\DeclareMathOperator{\id}{id}

\DeclareMathOperator{\Mon}{Mon}

\renewcommand{\S}{\mathbf{S}}
\renewcommand{\H}{\mathbf{H}}  
\newcommand{\Ein}{\mathbf{Ein}} 
\newcommand{\X}{\mathbf{X}}

\renewcommand{\epsilon}{\varepsilon}
\renewcommand{\geq}{\geqslant}
\renewcommand{\leq}{\leqslant}
\renewcommand{\hat}{\widehat}  
\newcommand{\hx}{\hat{x}}
\renewcommand{\tilde}{\widetilde}
\renewcommand{\bar}{\overline}
\title{Semi-simple Lie groups acting conformally on compact Lorentz manifolds}
\author{Vincent Pecastaing}

\begin{document}

\maketitle

\begin{center}
\today
\end{center}

\begin{abstract}
We give a classification, up to local isomorphisms, of semi-simple Lie groups without compact factors that can act faithfully and conformally on a compact Lorentz manifold of dimension greater than or equal to $3$.
\end{abstract}

\tableofcontents

\section{Introduction}

In the middle of the 1980's, Zimmer proved a deep result on differentiable actions of Lie groups that preserve some geometric structure, called \textit{Zimmer's embedding theorem} (\cite{zimmer86}, Theorem A). Let $G$ be a real linear algebraic group and $H$ a non-compact simple Lie group. Assuming that $H$ acts on a manifold $M$ so as to preserve a $G$-structure and a finite volume, Zimmer's theorem gives strong algebraic constraints on $H$, a noticeable one being that $H$ can be locally embedded into $G$. In the same paper, Zimmer gave a striking corollary to his theorem in the special case where the $G$-structure is defined by a Lorentz metric.

\begin{thm*}[\cite{zimmer86}, Theorem B]
Up to local isomorphism, $\PSL(2,\R)$ is the only non-compact simple Lie group that can act faithfully and isometrically on a Lorentz manifold of finite volume.
\end{thm*}

A compact Lorentz manifold admitting an isometric and faithful action of $H:=\PSL(2,\R)$ can be easily built. Let $g_K$ be the Killing metric of $H$. It has Lorentz signature and is invariant under left and right translations of $H$ on itself. Choose $\Gamma < H$ any uniform lattice, and set $M = H / \Gamma$. The action of $\Gamma$ on $(H,g_K)$ by right translations being isometric, $g_K$ induces a Lorentz metric $g$ on $M$. Moreover, the action of $H$ on itself by left translations is also isometric for $g_K$ and centralizes the right translations. Therefore, $H$ acts on $M = H / \Gamma$ by isometries of $g$.

Note that this situation contrasts with the analogous Riemannian situation: the isometry group of a compact Riemannian manifold is always a compact Lie group (it follows from Ascoli's theorem). Thus, what we observe here is that although preserving a Lorentzian metric tensor is less restrictive than preserving a Riemannian one, it is still a \textit{rigid} condition that allows us to classify groups of isometries. In fact, Zimmer's result led Adams, Stuck and - independently - Zeghib to the full classification up to local isomorphism of the isometry group of a compact Lorentz manifold, see \cite{adams}, \cite{zeghib}. Thus, we have a good understanding of the groups of isometries of compact Lorentz manifolds and we would like to know if it is possible to classify their conformal groups.

Let $(M,g)$ be a pseudo-Riemannian manifold of signature $(p,q)$. A diffeomorphism $f \in \Diff(M)$ is said to be \textit{conformal} with respect to $g$ if there exists a positive map $\varphi \in \mathcal{C}^{\infty}(M)$ such that $f^* g = \varphi g$. Naturally, conformal diffeomorphisms of $g$ act on $M$ preserving the \textit{conformal class} $[g] = \{\e^{\sigma}g, \ \sigma \in \mathcal{C}^{\infty}(M)\}$, and they form a group denoted $\Conf(M,[g])$. We call \textit{conformal structure} the data of a conformal class $[g]$ of pseudo-Riemannian metrics on a smooth differentiable manifold $M$.

One of the most remarkable properties of conformal structures is their \textit{rigidity} in dimension greater than or equal to $3$. We mean ``rigidity'' in Gromov's sense of rigid geometric structures (in \cite{gromov}). Roughly, the main idea behind Gromov's definition is that, a point $x_0 \in M$ being given, a local conformal diffeomorphism defined near $x_0$ is determined by its $2$-jet at $x_0$. In particular, such transformations cannot act trivially on a non-empty open set without being globally trivial and another important corollary of this rigidity phenomenon is that for any conformal structure $(M^n,[g])$, with $n \geq 3$, the group $\Conf(M,[g])$ is a Lie transformation group of dimension at most $\frac{(n+1)(n+2)}{2}$.

For every signature $(p,q)$, with $p+q \geq 3$, there exists a compact pseudo-Riemannian manifold whose conformal group has the maximal dimension: it is the so-called \textit{Einstein's universe} $\Ein^{p,q}$. It is the quotient $(\S^p \times \S^q) / \Z_2$, where $\Z_2$ acts by the product of the antipodal maps, endowed with the metric induced by $-g_{\S^p} \oplus g_{\S^q}$, where $g_{\S^k}$ denotes the Riemannian metric of $\S^k$ with curvature $+1$. It can be shown that $\Ein^{p,q}$ admits $\PO(p+1,q+1)$ as conformal group (see Section \ref{sss:einstein_cartan}). In Lorentzian signature, the conformal group of Einstein's universe is $\PO(2,n)$, showing that there are much more examples of conformal actions of simple Lie groups on compact Lorentz manifolds than isometric ones.

It is in fact possible to build a simpler compact Lorentz manifold admitting conformal actions of (smaller) simple Lie groups, namely $\SO(1,k)$, $k \geq 2$. Let $\R^{1,k}$ be the $(k+1)$-dimensional Minkowski space and $\Gamma = <2 \id>$ be the (conformal) group generated by a non-trivial homothety. Naturally, $\Gamma$ acts properly discontinuously on $\R^{1,k} \setminus \{0\}$ and is centralized by $\SO(1,k)$. Therefore, $\SO(1,k)$ acts conformally on the quotient $(\R^{1,k} \setminus \{0\}) / \Gamma$, usually called a \textit{Hopf manifold}.

The main result of this paper extends Zimmer's result on Lie groups of Lorentzian isometries to Lie groups of Lorentzian conformal diffeomorphisms, giving a classification of these groups up to local isomorphism. Although metric tensors and conformal classes of metrics are both rigid geometric structures in dimension at least $3$, \textit{a priori} there does not exist any finite measure invariant by some conformal Lie transformation group, and it would not be reasonable to assume that our groups act preserving a finite measure. Zimmer's embedding theorem rely mainly on ergodic results and it applies to the case of isometric Lie group actions on compact manifolds (these actions preserve the Riemannian volume, which is finite by compactness). Since we do not have the existence of a finite measure invariant under conformal transformations, the approach will change significantly when studying conformal Lie group actions.

In fact, Zimmer also gave a result that controls the real rank of a semi-simple Lie group without compact factor that acts on a compact manifold by preserving a $G$-structure (but not a finite measure). Thus, this result applies to the case of compact pseudo-Riemannian conformal structures (which are $CO(p,q)$-structures). In \cite{badernevo}, Bader and Nevo studied the situation where the bound is achieved. We have summarized their results in the following

\begin{thm*}[\cite{zimmer87},\cite{badernevo}]
Let $(M,[g])$ be a compact conformal pseudo-Riemannian structure of signature $(p,q)$, with $p+q \geq 3$ and $p \leq q$. Let $H$ be a semi-simple Lie group without compact factor that acts conformally on $M$. Then, $\Rk_{\R}(H) \leq p + 1$. Moreover, if $H$ is simple and $\Rk_{\R}(H) = p+1$, then $H$ is locally isomorphic to some $\SO(p+1,k)$, with $k \leq q+1$.
\end{thm*}

Thus, the situation is well described when the group $H$ has \textit{large} real rank. The aim of the present article is to complete this description in Lorentz signature, by specifying the missing semi-simple groups: those with \textit{small} real rank (in fact equal to $1$). We do it by exploiting a more recent result, stated in \cite{embed}, that extends Zimmer's embedding theorem to another class of geometric structures: \textit{Cartan geometries} (the reader not familiar with these structures will find a brief introduction in Section \ref{ss:cartan_general}). Together with Bader-Nevo's and Zimmer's theorems, we obtain the following classification theorem, which is the main result of this paper. 

\begin{thm}
\label{thm:main}
Let $H$ be a semi-simple Lie group without compact factor. Assume that $H$ acts faithfully and conformally on a compact Lorentz manifold $(M^n,g)$, with $n \geq 3$. Then, its Lie algebra $\h$ is isomorphic to one of the following Lie algebras:
\begin{itemize}
\item $\o(1,k)$, with $2 \leq k \leq n$,
\item $\mathfrak{su}(1,k)$, with $2 \leq k \leq \frac{n}{2}$,
\item $\o(2,k)$, with $2 \leq k \leq n$,
\item $\o(1,k) \oplus \o(1,s)$, with $k,s \geq 2$ and $k+s \leq \max(n,4)$.
\end{itemize}
Conversely, for each Lie algebra $\h$ in this list, there exists a Lie group $H$ with Lie algebra $\h$ and an $n$-dimensional compact Lorentz manifold $(M,g)$ such that $H \hookrightarrow \Conf(M,[g])$.
\end{thm}

In fact, the Lie algebras in the above list are precisely all the Lie subalgebras of $\o(2,n)$ which are semi-simple without compact factor. Thus, the previous result says that any semi-simple Lie group without compact factor that acts conformally on a compact Lorentz manifold can be locally embedded into the conformal group of the Lorentzian Einstein's universe of same dimension.

Theorem \ref{thm:main} shows that the conformal group of any compact Lorentz manifold $(M^n,g)$ is ``close to'' the one of $\Ein^{1,n-1}$. Naturally, one can ask a dual question: to what extent is \textit{the geometry} of $(M,[g])$ related to the one of $\Ein^{1,n-1}$ ? Similarly to what we have exposed before, this problem is well understood when the manifold admits conformal actions of semi-simple Lie groups with \textit{maximal} real rank. In \cite{franceszeghib}, Frances and Zeghib showed that if $(M,[g])$ is compact, with signature $(p,q)$ and admits a faithful conformal action of $\SO(p+1,k)$, with $k \geq p+1$, then it is conformally diffeomorphic to some quotient $\Gamma \backslash \tilde{\Ein^{p,q}}$ where $\Gamma$ is a discrete subgroup of $\Conf(\tilde{\Ein^{p,q}})$. More generally, the authors of \cite{embed} gave a bound on the real-rank of semi-simple Lie groups that act by automorphisms of a compact parabolic geometry and they proved a similar geometric result when this bound is achieved.

Thus, in Lorentz signature, the remaining question is to describe the geometry of compact manifolds admitting conformal actions of groups with real rank $1$. We leave this geometric aspect for further investigations.

\subsection*{Conventions and notations}

In this paper, ``manifold'' will mean a smooth differentiable manifold and all the objects considered will be assumed smooth. As usual, we will use the fraktur font to denote the Lie algebra of a Lie group. If $M$ is a manifold, we note $\mathfrak{X}(M)$ the Lie algebra of vector fields defined on $M$. A faithful and differentiable action of a Lie group $H$ on a manifold $M$ gives rise to a Lie algebra embedding $\h \hookrightarrow \mathfrak{X}(M)$ (to $X \in \h$ corresponds the infinitesimal generator of $\{\e^{-tX}\}_{t \in \R}$, seen as a flow on $M$). We will often identify $\h$ with this Lie algebra of vector fields.

\section{Background}

One of the great difficulties when working with groups of conformal transformations is that \textit{a priori}, these groups act without preserving any linear connection on the manifold. Thus, we cannot use standard tools of Riemannian
geometry (which are associated to the Levi-Civita connection of the metric) and we have to consider \textit{higher order} geometric objects to observe the rigidity of conformal structures.

A nice way to do so is to interpret a conformal structure $(M^n, [g])$, $n \geq 3$, as the data of a \textit{normalized Cartan geometry} modeled on the Einstein universe with same signature. The first purpose of this section is to recall briefly what the latter sentence means. The central tool of the proof of Theorem \ref{thm:main} is an embedding result, formulated in the framework of (general) Cartan geometries, given by Bader, Frances and Melnick in \cite{embed}, Theorem 1. We will also re-state this result.

\subsection{Cartan geometries and the equivalence problem for conformal structures}
\label{ss:cartan_general}

The notion of Cartan geometry, introduced by É. Cartan, can be understood as an idea of curved versions of homogeneous spaces, in the same sense than Riemannian manifolds are curved versions of the Euclidian space. Thus, a Cartan geometry is always defined with respect to a homogeneous space $\X=G/H$, called the \textit{model space} of the geometry. We start giving the general definitions and properties we need concerning these these geometric structures. The reader can find a deeper introduction in \cite{sharpe} or \cite{cap}.

\subsubsection{General definitions}

Let $G$ be a Lie group, $P < G$ a closed subgroup and $\X = G/P$ the corresponding homogeneous space.

\begin{dfn}
\label{dfn:cartan}
Let $M$ be a manifold. A \textit{Cartan geometry} $(M,\mathcal{C})$ modeled on $\X$ is the data  of $(M,\hat{M},\omega)$, where $\hat{M} \rightarrow M$ is a $P$-principal fiber bundle and $\omega$ is a $\g$-valued $1$-form on $\hat{M}$, called the \textit{Cartan connection}, satisfying the following conditions:
\begin{enumerate}
\item For all $\hx \in \hat{M}$, $\omega_{\hx} : T_{\hx}\hat{M} \rightarrow \g$ is a linear isomorphism ;
\item For all $X \in \p$, $\omega(X^*) = X$, where $X^*$ denotes the fundamental vector field associated to the right action of $\exp(tX)$ ;
\item For all $p \in P$, $(R_p)^*\omega = \Ad(p^{-1})\omega$.
\end{enumerate}
\end{dfn}

Note that the definition implies $\dim M = \dim \X$. If $(M_1,\mathcal{C}_1)$ and $(M_2,\mathcal{C}_2)$ are two Cartan geometries modelled on $\X$, a \textit{morphism} between them is naturally a morphism of $P$-principal bundles between the Cartan bundles $F : \hat{M_1} \rightarrow \hat{M_2}$, which respects the Cartan connections, \textit{i.e.} $F^* \omega_2 = \omega_1$. Thus every morphism $F : \hat{M_1} \rightarrow \hat{M_2}$ is over a map $f : M_1 \rightarrow M_2$.

When the model space is \textit{effective} (\textit{i.e.} when the natural action $P \curvearrowright \X$ by left translations is faithful), the base map $f$ completely determines $F$: there is at most one bundle morphism $F : \hat{M_1} \rightarrow \hat{M_2}$ that covers $f$ and such that $F^* \omega_2 = \omega_1$ (see \cite{cap}, Prop. 1.5.3).

Thus, when the model space is effective, we can consider a morphism of Cartan geometries as a map between the bases of the Cartan bundles.

\begin{dfn}
Let $\X$ be an effective homogeneous space, and $(M_1,\mathcal{C}_1)$, $(M_2,\mathcal{C}_2)$ be two Cartan geometries modelled on $\X$. A local diffeomorphism $f : M_1 \rightarrow M_2$ is said to be a morphism of Cartan geometries if it is covered by a bundle morphism between the Cartan bundles which respects the Cartan connections. Such a bundle morphism is necessarily unique. We call it \textit{the lift of} $f$, and note it $\hat{f}$.
\end{dfn}

In what follows, we will only consider effective model spaces. We note $\Aut(M,\mathcal{C})$ the group of isomorphisms of $(M,\mathcal{C})$ into itself.

\begin{dfn}[Infinitesimal automorphisms]
A vector field $X \in \mathfrak{X}(M)$ is said to be a Killing vector field of $(M,\mathcal{C})$ if its local flow $\phi_X^t$ is composed with local automorphisms of $(M,\mathcal{C})$.
\end{dfn}

This last definition is in fact equivalent to the existence of a unique vector field $\hat{X} \in \mathfrak{X}(\hat{M})$ such that 
\begin{itemize}
\item $\hat{X}$ commutes to the right $P$-action on $\hat{M}$ ;
\item $\pi_* \hat{X} = X$ ;
\item $\mathcal{L}_{\hat{X}} \omega = 0$.
\end{itemize}
We also call $\hat{X}$ the lift of the Killing vector field $X$.

\subsubsection{Basic properties}
\label{sss:basic_properties}

We collect here some classical or elementary facts concerning Cartan geometries and their automorphisms. We consider $(M,\mathcal{C})$ a Cartan geometry with model space $\X = G/P$, and we note $\pi : \hat{M} \rightarrow M$ its Cartan bundle and $\omega \in \Omega^1(\hat{M}, \g)$ its Cartan connection.

\vspace*{0.2cm}

\textbf{Automorphisms group.} The group $\Aut(M,\mathcal{C})$ admits a unique differential structure which makes it a Lie transformation group. The Lie algebra of this group is naturally identified with the Lie algebra of complete Killing vector fields of $(M,\mathcal{C})$. By definition, the Lie group $\Aut(M,\mathcal{C})$ acts on $\hat{M}$ (via the lifts $\hat{f}$), and this late action is free and proper. Consequently if $X \in \Kill(M,\mathcal{C})$, its lift $\hat{X}$ never vanishes on $\hat{M}$ unless $X=0$.

\vspace*{0.2cm}

\textbf{Generalizations from Riemannian geometry.} One of the most interesting aspects of Cartan geometries is that it is possible to define familiar notions of differential geometry such as curvature, torsion, tensors or covariant derivative, generalizing the ones of a standard linear connection (see \cite{sharpe}, §5.3). These objects have properties similar to Riemannian geometry: for instance, when the curvature vanishes identically, the Cartan geometry yields the structure of a $(G,\X)$-manifold on $M$.

\subsubsection{The equivalence problem}

Sometimes, the data of a Cartan geometry (with some model space) is equivalent to the one of a ``classic'' geometric structure. For instance, when $\X = \Aff(\R^n) / \GL(\R^n)$ is the standard affine space, the data of a torsion-free Cartan geometry $(M,\mathcal{C})$ modeled on $\X$ is equivalent to the one of a torsion-free affine connection $\nabla$ on $M$. Another remarkable example is that when $\X = \Isom(\R^{p,q})/ O(p,q)$ is the standard pseudo-Euclidian space, the data of a torsion-free Cartan geometry $(M,\mathcal{C})$ modeled on $\X$ is equivalent to the one of a pseudo-Riemannian metric $g$ of signature $(p,q)$ on $M$ (this last assertion is nothing more than the existence and uniqueness of the Levi-Civita connection defined by a metric).

The problem that consists in finding a homogeneous space $\X$ such that there is an equivalence of categories between Cartan geometries with model space $\X$ and manifolds endowed with some classic geometric structure is called the \textit{equivalence problem}. This problem has been solved for most of the familiar examples of rigid geometric structures (pseudo-Riemannian metrics, conformal structures in dimension at least $3$, non-degenerate CR-structures..). Thus, in lots of geometric contexts, the theory of Cartan geometries yields formalism and materials to analyse problems.

This point of view is particularly relevant when working with conformal pseudo-Riemannian structures. The appropriate homogeneous space $\X$ used to interpret a conformal structure in terms of a Cartan geometry is the \textit{Einstein universe} with same signature.

\subsubsection{The model space of conformal geometry and Cartan's theorem}
\label{sss:einstein_cartan}

Let $p,q \geq 0$ be two integers such that $n:=p+q \geq 3$. Consider $\R^{p+1,q+1}$ the vector space $\R^{n+2}$ endowed with the non-degenerate quadratic form $q(x) = -x_1^2 - \cdots -x_{p+1}^2 + x_{p+2}^2 + \cdots + x_{n+2}^2$ with signature $(p+1,q+1)$. Note $\pi : \R^{n+2} \setminus \{0\} \rightarrow \R P^{n+1}$ the natural projection.

\begin{dfn}
We define (as a manifold) the \textit{Einstein universe} of signature $(p,q)$ as being the quadric in $\R P^{n+1}$ defined by $\Ein^{p,q} = \pi(\{q=0\} \setminus \{0\})$.
\end{dfn}

Thus, $\Ein^{p,q}$ is an $n$-dimensional smooth compact projective variety, on which the group $\PO(p+1,q+1)$ acts transitively. Consequently, as a homogeneous space, $\Ein^{p,q}$ can be identified with a quotient $\PO(p+1,q+1) / P$, where $P < \PO(p+1,q+1)$ is a parabolic Lie subgroup.

The quadric $\Ein^{p,q}$ naturally inherits a conformal class $[g_{p,q}]$ of non-degenerate metrics of signature $(p,q)$ from the ambient quadratic form of $\R^{p+1,q+1}$, and this conformal class is invariant under the action of $\PO(p+1,q+1)$. In fact, we have $\Conf(\Ein^{p,q}) = \PO(p+1,q+1)$.

When the signature is Riemannian, we have a conformal identification $\Ein^{0,n} \simeq \S^n$ with the standard Riemannian sphere with same dimension. We can now state the following central result, due to É. Cartan in the Riemannian case (see \cite{cartan23}).

\begin{thm*}[Equivalence problem for conformal structures, É. Cartan]
Let $p,q \geq 0$ be two integers such that $n:=p+q \geq 3$. There is an equivalence of categories
\begin{equation*}
\left \{
\begin{array}{c}
\text{Cartan geometry } (M^n,\mathcal{C}) \\
\text{ with model space } \Ein^{p,q} \\
\text{which are normalized}
\end{array}
\right \}
\leftrightarrow
\left \{
\begin{array}{c}
\text{Conformal structures } (M^n,[g]) \\
\text{where } [g] \text{ is a conformal class} \\
\text{of metrics of signature } $(p,q)$
\end{array}
\right \}.
\end{equation*}
\end{thm*}

The normalization assumption in this last theorem is a technical condition on the Cartan connection $\omega$ of a Cartan geometry modeled on $\Ein^{p,q}$ which is made in order to ensure the uniqueness of the Cartan geometry associated to a conformal structure.

This result is known to be technical and not immediate. The reader can find an exposition of the construction of the normalized Cartan geometry associated to $(M,[g])$ in \cite{cap} (chapter 1.6) or \cite{koba} (chapter IV.4). Nevertheless, it is more simple to see how is built the conformal class associated to a Cartan geometry modeled on $\Ein^{p,q}$. Since it will have a significant importance in the proof of Theorem \ref{thm:main}, we explain it briefly.

\subsubsection{The easy part of Cartan's theorem}
\label{sss:cartan_easy}

Generally, if $(M,\mathcal{C})$ is any Cartan geometry modeled on a homogeneous space $G/P$, with Cartan bundle $\pi : \hat{M} \rightarrow M$, the Cartan connection $\omega$ identifies the tangent bundle $TM \simeq \hat{M} \times_P \g / \p$, where $P$ acts on $\g / \p$ via its adjoint representation on $\g$ (see \cite{sharpe}, p.188). Precisely, we have a family of identifications $\varphi_{\hx} : T_xM \rightarrow \g / \p$, parametrized by $\hx \in \pi^{-1}(x)$, which satisfy the equivariancy relation $\varphi_{\hx.p} = \bar{\Ad}(p^{-1}) \varphi_{\hx}$. These maps are built in the following way: if $v \in T_xM$ is a tangent vector and $\hx$ is in the fiber over $x$, choose any $\hat{v} \in T_{\hx}\hat{M}$ such that $\pi_* \hat{v} = v$. Then, set $\varphi_x(v) : = \omega_{\hx}(\hat{v}) + \p \in \g / \p$. This definition does not depend on the choice of $\hat{v}$ because of the second axiom of Cartan geometries, and the equivariancy comes from the third axiom.

In the situation where $G/P = \Ein^{p,q}$,  with $G = \PO(p+1,q+1)$ and $P=\Stab_G(x_0)$ for some $x_0 \in \Ein^{p,q}$, then $\g / \p$ is identified with the tangent space $T_{x_0}\Ein^{p,q}$ and the adjoint representation $\bar{\Ad} : P \rightarrow \GL(\g / \p)$ corresponds to the isotropy representation $\rho_{x_0} : P \rightarrow \GL(T_{x_0} \Ein^{p,q})$. Therefore, the conformal class of signature $(p,q)$ of $\Ein^{p,q}$ yields a conformal class $[Q] = \{\lambda Q, \ \lambda \in \R_+^*\}$ of  quadratic forms of signature $(p,q)$ on $\g / \p$ which is invariant under the adjoint action $\bar{\Ad}(P) < \GL(\g / \p)$. Thus, if $(M,\mathcal{C})$ is modeled on $\Ein^{p,q}$, we can provide $M$ a conformal conformal class of metrics of signature $(p,q)$ by setting $\forall x \in M, \ [g_x] := \varphi_{\hx}^* [Q]$ for any $\hx$ in the fiber over $x$ (this is well defined by $\bar{\Ad}(P)$-invariance of $[Q]$).

\subsection{A more recent embedding theorem}
\label{ss:embed}

The main result of the present work rely on a formulation for Cartan geometries of Zimmer's embedding theorem. We recall here the content of this result.

Let $(M,\mathcal{C})$ be a Cartan geometry, modeled on $\X = G/P$. We note $\pi : \hat{M} \rightarrow M$ its Cartan bundle and $\omega \in \Omega^1(\hat{M},\g)$ its Cartan connection.

The context is the following: we are given a faithful Lie group action $H \times M \rightarrow M$ on a manifold $M$, and we assume that this action preserves the Cartan geometry $(M,\mathcal{C})$. This global $H$-action gives rise to a Lie algebra embedding $\h \hookrightarrow \Kill(M,\mathcal{C})$ and from now on, $\h$ will denote the corresponding Lie subalgebra of Killing vector fields of $(M,\mathcal{C})$. Recall that if $X$ is a Killing vector field of $(M,\mathcal{C})$, we denote by $\hat{X} \in \mathfrak{X}(\hat{M})$ its lift to $\hat{M}$.

For all $\hx \in \hat{M}$, we now define a linear morphism $\iota_{\hx} : \h \rightarrow \g$ by setting 
\begin{equation*}
\forall X \in \h, \ \iota_{\hx}(X) = \omega_{\hx}(\hat{X}_{\hx}).
\end{equation*}
According to section \ref{sss:basic_properties}, $\iota_{\hx}$ is injective (but it is \textit{a priori} not a Lie algebra embedding). Thus, we are given a differentiable map
\begin{align}
\label{align:iota}
\iota : & \ \hat{M} \rightarrow \Mon(\h,\g) \\
        & \ \ \hx \ \mapsto \iota_{\hx} , \nonumber
\end{align}
where $\Mon(\h,\g)$ denotes the variety of injective linear morphisms $\h \rightarrow \g$. Moreover,  the product group $H \times P$ naturally acts (on the left) on $\hat{M}$ via $(h,p).\hx = h.\hx.p^{-1}$ (recall that the $H$-action commutes to the $P$-action). A simple computation shows that
\begin{equation*}
\iota((h,p).\hx) = \Ad(p) \circ \iota(\hx) \circ \Ad(h^{-1}),
\end{equation*}
\textit{i.e.} the map $\iota : \hat{M} \rightarrow \Hom(\h,\g)$ is $(H \times P)$-equivariant when $H \times P$ acts on $\Hom(\h,\g)$ via $(h,p).\alpha = \Ad(p) \circ \alpha \circ \Ad(h^{-1})$.

\subsubsection{Zimmer points of a measurable action}

\begin{dfn}[Zimmer points]
Let $S$ be a Lie subgroup of $H$. A point $x \in M$ is said to be a \textit{Zimmer point} for $S$ if for some (equivalently for all) $\hx \in \pi^{-1}(x)$, we have $S.\iota(\hx) \subset \iota(\hx).P$.
\end{dfn}

The following proposition shows that the existence of Zimmer point enable us to relate the Lie group $H$ with the model space $G/P$.

\begin{prop}[\cite{embed}]
\label{prop:morphisme_surjectif}
Let $x \in M$ be a Zimmer point for $S$. For all $\hx \in \pi^{-1}(x)$, there exists an algebraic subgroup $P^{\hx} < P$ and a surjective algebraic morphism $\rho_{\hx} : P^{\hx} \rightarrow \Ad_{\h}(S)$.
\end{prop}

\begin{proof}
By definition, $\forall s \in S$, $\exists p \in P$ such that $\iota(\hx) \circ \Ad(s^{-1}) = \Ad(p) \circ \iota(\hx)$. Set $\h^{\hx} := \iota_{\hx}(\h) \subset \g$ (it is only a linear subspace). Let $P' < P$ be the stabilizer $P' = \{p \in P \ | \ \Ad(p) \text{ stabilizes } \h^{\hx} \}$. Define $\rho_{\hx} : P' \rightarrow \GL(\h)$ by setting $\rho_{\hx}(p) = \iota_{\hx}^{-1} \circ \Ad_{\h}(p) \circ \iota_{\hx}$. Set $P^{\hx} = \rho_{\hx}^{-1}(\Ad_{\h}(S))$. The fact that $x$ is a Zimmer point simply says that $\rho_{\hx} : P^{\hx} \rightarrow \Ad_{\h}(S)$ is onto. 
\end{proof}

As we see in the previous proof, the morphism $\rho_{\hx}$ is obtained by restricting the adjoint action of a certain subgroup $P^{\hx} < P$ to the subspace $\iota_{\hx}(\h) \subset \g$. We can summarize this by saying that the adjoint action of $S$ on $\h$ is ``contained'' in the adjoint action of $P$ on $\h$:
\begin{equation*}
\Ad_{\g}(P^{\hx}) =
\begin{pmatrix}
\Ad_{\h}(S) & * \\
0 & *
\end{pmatrix},
\end{equation*}
and thus that we can find algebraic properties of the couple $(H,S)$ inside the couple $(G,P)$. As we will see during the proof of Theorem \ref{thm:embed}, it is also possible to derive geometric informations from the morphism $\rho_{\hx}$ (see in particular Section \ref{sss:general_zimmer_conformal}).

The embedding theorem of Bader, Frances and Melnick gives sufficient conditions on $S$ that ensure the existence of Zimmer points. Before stating it, we give the following definition.

\begin{dfn}
A Lie subgroup $S < H$ is said to be \textit{discompact} if the Zariski closure $\bar{\Ad_{\h}(S)}^{\text{Zar}} \subset \GL(\h)$ admits no proper algebraic normal cocompact subgroup.
\end{dfn}

Examples of discompact subgroups are $S < H$ such that the Zariski closure of $\Ad_{\h}(S)$ is an hyperbolic torus (think to the adjoint action of a Cartan subgroup of a semi-simple Lie group) or a unipotent subgroup (think to the adjoint action of some positive root spaces).

We can now state the theorem.

\begin{thm}[\cite{embed}]
\label{thm:embed}
Let $(M,\mathcal{C})$ be a Cartan geometry with model space $\X = G/P$. Let $H$ be a Lie group acting on $M$ by automorphisms, and $S < H$ a Lie subgroup. Assume that $S$ is discompact and acts on $M$ preserving a finite measure $\mu$. Then, $\mu$-almost every point $x \in M$ is a Zimmer point for $S$.
\end{thm}

Typically, if $S$ is discompact and amenable, and if the manifold is compact, the hypothesis of Theorem \ref{thm:embed} are satisfied and there must exist a Zimmer point for $S$. As we shall see, the existence of Zimmer points is a very rich information that will be sufficient to prove Theorem \ref{thm:main}.

\section{Proof of the main theorem}

Theorem \ref{thm:main} classifies semi-simple Lie groups without compact factor that can act by conformal diffeomorphisms on a compact Lorentz manifold. The result of Zimmer cited in the introduction reduces our work to semi-simple Lie group of real rank at most $2$. When the group has $\R$-rank $2$, the situation is already well described and an algebraic lemma will be enough to obtain the conclusion of Theorem \ref{thm:main}. We postpone it to the last Section \ref{ss:rank2}. Therefore, the substantial contribution of our work is the description of non-compact simple Lie groups of $\R$-rank $1$ acting conformally on compact Lorentz manifolds.

The non-compact simple Lie algebras of real rank $1$ are: $\o(1,k)$, $k \geq 2$, $\su(1,k)$, $k \geq 2$, $\mathfrak{sp}(1,k)$, $k \geq 2$ and $\mathfrak{f}_4^{-20}$. Section \ref{ss:preliminaries} gives elementary properties of them that we will use. We then prove in Section \ref{ss:exclusion} that no group locally isomorphic to $\Sp(1,k)$ or $F_4^{-20}$ appears in our classification. In Section \ref{ss:orbits}, we give geometric properties of the orbits of certain Zimmer points of a conformal action of a Lie group locally isomorphic to $\SO(1,k)$ or $\SU(1,k)$; and we derive from these properties that $k$ is bounded by the dimension (or half the dimension) of the manifold.

\subsection{Preliminary description of non-compact simple real Lie algebras of real rank $1$}
\label{ss:preliminaries}
\subsubsection{Definitions of orthogonal, special unitary and symplectic rank $1$ simple Lie algebras}
\label{sss:preliminaries_orthogonal}

Let $n \geq 3$. We note
\begin{equation*}
J_{1,n-1} = 
\begin{pmatrix}  
0 & & 1 \\
 & I_{n-2} & \\
1& & 0
\end{pmatrix}.
\end{equation*}
We denote by $\H$ the non-commutative field of quaternions. If $M$ is a matrix with complex or quaternionic coefficients, we note $M^* = \! ~^{t} \bar{M}$.

We define for $n \geq 3$
\begin{align*}
\o(1,n-1) & = \{M \in \gl_n(\R) \ | \ M^* J_{1,n-1} + J_{1,n-1} M = 0\} \\
\su(1,n-1) & = \{ M \in \gl_n (\C) \ | \ \Tr M = 0 \text{ and } M^* J_{1,n-1} + J_{1,n-1}M = 0 \} \\
\sp(1,n-1) & = \{ M \in \gl_n(\H) \ | \ M^* J_{1,n-1} + J_{1,n-1}M = 0 \}.
\end{align*}
All of them are real non-compact simple Lie algebras of real rank $1$ admitting the following root-space decompositions:
\begin{align*}
\o(1,n-1) & = \h_{-\alpha} \oplus \a \oplus \o(n-2) \oplus \h_{\alpha}, \text{ with } \dim \h_{\alpha} = n-2 \\
\su(1,n-1) & = \h_{-2\alpha} \oplus \h_{-\alpha} \oplus \a \oplus \u(n-2) \oplus \h_{\alpha} \oplus \h_{2\alpha}, \text{ with } \dim \h_{\alpha} = 2n-4 \text { and } \dim \h_{2\alpha} = 1\\
\sp(1,n-1) & = \h_{-2\alpha} \oplus \h_{-\alpha} \oplus \a \oplus \o(3) \oplus \sp(n-2) \oplus \h_{\alpha} \oplus \h_{2\alpha}, \text{ with } \dim \h_{\alpha} = 4n -8 \\ & \text { and } \dim \h_{2\alpha} = 3,
\end{align*}
where the factor $\a$ always denotes a natural Cartan subspace (for the Cartan involution $\theta : X \mapsto -X^*$), $\alpha$ a simple restricted root and $\h_{\pm \alpha}$ or $\h_{\pm 2\alpha}$ the corresponding root spaces. We can simply observe these decompositions using the matrix representations. For instance, a matrix $M \in \gl(n,\R)$ belongs to $\o(1,n-1)$ if and only if it has the form
\begin{equation*}
M = 
\begin{pmatrix}
 a & u & 0 \\
 - ^{t} v & N & - ^{t} u \\ 
 0 & v & -a
\end{pmatrix}
,
\text{ with } a \in \R, \ u,v \in \R^{n-2} \text{ and } N \in \o(n-2).
\end{equation*}
The Cartan subspace $\a$ corresponds to the ($1$-dimensional) space of diagonal $\R$-split matrices and the associated root spaces are the following subspaces of upper (resp. lower) triangular matrices
\begin{equation*}
\h_{\alpha} = 
\left \lbrace 
\begin{pmatrix}
0 & u & 0 \\
  & 0 & - ^{t} u \\ 
  &   & 0
\end{pmatrix}
, \ u \in \R^{n-2}
\right \rbrace
\text{ and }
\h_{-\alpha} =
\left \lbrace 
\begin{pmatrix}
0 &   &  \\
- ^{t} v  & 0 &  \\ 
0  & v & 0
\end{pmatrix}
, \ v \in \R^{n-2}
\right \rbrace.
\end{equation*}
\subsubsection{Classical and quaternionic Heisenberg Lie algebras.}
\label{sss:heisenberg}

The Lie algebras $\su(1,n-1)$ and $\sp(1,n-1)$ contain nilpotent Lie subalgebras, with degree of nilpotence $2$, that will play an important role in the proof of Theorem \ref{thm:main}. We define them intrinsically, and then identify them in $\su(1,n-1)$ and $\sp(1,n-1)$ using the root-space decompositions.

\begin{dfn}
Let $k \geq 1$. We note $\omega_{\C}$ the standard symplectic form on $\C^k$ and $\omega_{\H}$ the $\H$-valued alternate $2$-form on $\H^k$ given by $\omega_{\H}(u,v) = u ^{t}\!\bar{v} - v ^{t}\!\bar{u}$ for $u,v \in \H^k$.

We define the \textit{classical Heisenberg Lie algebra} $\heis^{\C}(2k+1)$ of dimension $2k +1$ as the real vector space $\C^k \oplus \R.Z$, the element $Z$ being outside $\C^k$, endowed with the Lie algebra structure whose center is exactly $\R.Z$ and such that if $z,z' \in \C^k$, $[z,z'] = \omega_{\C}(z,z')Z$.

We define the \textit{quaternionic Heisenberg Lie algebra} $\heis^{\H}(4k+3)$ of dimension $4k+3$ as the real vector space $\H^k \oplus \Span_{\R}(Z_i,Z_j,Z_k)$, the elements $Z_i,Z_j,Z_k$ being linearly independent and outside $\H^k$, endowed with the Lie algebra structure whose center is exactly $\Span_{\R}(Z_i,Z_j,Z_k)$ and such that if $q,q' \in \H^k$, if we write $\omega_{\H}(q,q') = a.i + b.j +c.k$, then we have $[q,q'] = aZ_i + bZ_j + cZ_k$.
\end{dfn}

The commutator ideals of these Lie algebras coincide with their respective centers. Thus, they are nilpotent Lie algebra with degree of nilpotence $2$.

\begin{lem}
\label{lem:heis_su_sp}
Let $n \geq 3$. The Lie algebras $\su(1,n-1)$ and $\sp(1,n-1)$ contain $\heis^{\C}(2n-3)$ and $\heis^{\H}(4n-5)$ respectively as subalgebras.
\end{lem}

\begin{proof}
In $\su(1,n-1)$, the root spaces of positive roots correspond to the subspaces of strictly upper-triangular matrices
\begin{equation*}
\h_{\alpha} = 
\left \lbrace 
\begin{pmatrix}
0 & u & 0 \\
  & 0 & - ^{t} \bar{u} \\ 
  &   & 0
\end{pmatrix}
, \ u \in \C^{n-2}
\right \rbrace
\text{ and }
\h_{2\alpha} = 
\left \lbrace 
\begin{pmatrix}
0 & 0 & ic \\
  & 0 & 0 \\ 
  &   & 0
\end{pmatrix}
, \ c \in \R
\right \rbrace
\end{equation*}
and the bracket $\h_{\alpha} \times \h_{\alpha} \rightarrow \h_{2\alpha}$ is given by
\begin{equation*}
\forall u,v \in \C^k , \
\left [
\begin{pmatrix}
0 & u & 0 \\
  & 0 & - ~^{t} \bar{u} \\
  & & 0
\end{pmatrix}
,
\begin{pmatrix}
0 & v & 0 \\
  & 0 & - ~^{t} \bar{v} \\
  & & 0
\end{pmatrix}
\right ]
=
\begin{pmatrix}
0 & 0 & i \omega_{\C}(u,v) \\
 & 0 & 0 \\
 & & 0
\end{pmatrix}
.
\end{equation*}
Thus, the subspace $\h_{\alpha} \oplus \h_{2\alpha} \subset \su(1,n-1)$ is a Lie subalgebra isomorphic to $\heis^{\C}(2n-3)$.

\vspace*{0.2cm}

In $\sp(1,n-1)$ the root spaces of positive roots correspond to the subspace of strictly upper-triangular matrices
\begin{equation*}
\h_{\alpha} = 
\left \lbrace 
\begin{pmatrix}
0 & u & 0 \\
  & 0 & - ^{t} \bar{u} \\ 
  &   & 0
\end{pmatrix}
, \ u \in \H^{n-2}
\right \rbrace
\text{ and }
\h_{2\alpha} = 
\left \lbrace 
\begin{pmatrix}
0 & 0 & q \\
  & 0 & 0 \\ 
  &   & 0
\end{pmatrix}
, \ q \in \Span_{\R}(i,j,k)
\right \rbrace
\end{equation*}
and the Lie bracket $\h_{\alpha} \times \h_{\alpha} \rightarrow \h_{2\alpha}$ is given by
\begin{equation*}
\forall u,v \in \H^k , \
\left [
\begin{pmatrix}
0 & u & 0 \\
  & 0 & - ~^{t} \bar{u} \\
  & & 0
\end{pmatrix}
,
\begin{pmatrix}
0 & v & 0 \\
  & 0 & - ~^{t} \bar{v} \\
  & & 0
\end{pmatrix}
\right ]
=
\begin{pmatrix}
0 & 0 & \omega_{\H}(u,v) \\
 & 0 & 0 \\
 & & 0
\end{pmatrix}
.
\end{equation*}
Thus, the Lie subalgebra $\h_{\alpha} \oplus \h_{2\alpha} \subset \su(1,n-1)$ is isomorphic to $\heis^{\H}(4n-5)$.
\end{proof}

\subsubsection{The exceptional real rank $1$ simple Lie algebra}

There exists an exceptional real rank $1$ simple Lie algebra, noted $\mathfrak{f}_4^{-20}$. Our work does not need a complete description of it, we will just use the following lemma.

\begin{lem}
\label{lem:heis_f4}
The Lie algebra $\mathfrak{f}_4^{-20}$ contains a subalgebra isomorphic to $\heis^{\H}(7)$.
\end{lem}

\begin{proof}
The roots system of $\mathfrak{f}^4_{-20}$ has also type $(BC)_1$, \textit{i.e.} we have a simple root $\alpha$ such that the restricted roots are $\{\pm \alpha, \pm 2\alpha\}$ (\cite{knapp}, p.717). We keep the same notations for the root-spaces. Similarly with the previous situations, $\h_{\alpha} \oplus \h_{2\alpha}$ is a unipotent subalgebra and we have $\dim \h_{\alpha} = 8$, $\dim \h_{2\alpha}=7$. Moreover, there is an identification $\h_{\alpha} \simeq \mathbf{O}$, algebra of octonions, and an identification of $\h_{2\alpha}$ with purely imaginary octotions, such that the Lie bracket $\h_{\alpha} \times \h_{\alpha} \rightarrow \h_{2\alpha}$ corresponds to $(x_1,x_2) \in \mathbf{O} \times \mathbf{O} \mapsto x_1 \sigma(x_2) - x_2\sigma(x_1)$, where $\sigma$ denotes the conjugation in $\mathbf{O}$.

The Cayley-Dickson construction shows that the octonions can be defined as pairs of quaternions endowed with a product built with products and conjugations of quaternions. It is thus not difficult to see (finding an appropriate copy of $\H$ in $\h_{\alpha}$) that $\h_{\alpha}$ contains a $4$-dimensional subspace $\h_{\alpha}' \subset \h_{\alpha}$ such that $\h_{\alpha}' \oplus [\h_{\alpha}',\h_{\alpha}'] \simeq \heis^{\H}(7)$.
\end{proof}

\subsection{Exclusion of the Lie algebras $\mathfrak{sp}(1,k)$, $k \geq 2$, and $\mathfrak{f}_4^{-20}$}
\label{ss:exclusion}

Let $H$ be a Lie group locally isomorphic to some $\Sp(1,k)$, $k \geq 2$ or $F_4^{-20}$. The aim of this section is to prove that there never exists a compact Lorentz manifold $(M,g)$, with $\dim M \geq 3$, on which $H$ acts faithfully and conformally.

We fix a compact manifold $(M^n,[g])$, with $n \geq 3$, endowed with a conformal class of Lorentz metrics $[g]$. Let $(M,\mathcal{C})$ be the corresponding normalized Cartan geometry modeled on $\Ein^{1,n-1}$ given by Cartan's theorem (section \ref{sss:einstein_cartan}). We note $G = \PO(2,n)$ and $P < G$ the stabilizer of an isotropic line in $\R^{2,n}$, so that we can identify $\Ein^{1,n-1} \simeq G/P$. We assume that $H < \Conf(M,[g])$, or equivalently that $H$ acts faithfully on $M$ by automorphisms of the Cartan geometry $(M,\mathcal{C})$.

\subsubsection{An embedding given by Proposition \ref{prop:morphisme_surjectif}}

For any such $H$, be it locally isomorphic to $Sp(1,k)$ or to $F_4^{-20}$, its root-space decomposition has the form $\h = \h_{-2\alpha} \oplus \h_{-\alpha} \oplus \h_0 \oplus \h_{\alpha} \oplus \h_{2\alpha}$ and there exists a subalgebra $\s \subset \h_{\alpha} \oplus \h_{2\alpha}$ isomorphic to $\heis^{\H}(7)$ (Lemmas \ref{lem:heis_su_sp} and \ref{lem:heis_f4}). We choose and fix such an $\s$, and note $S$ the corresponding connected Lie subgroup of $H$. 

The linear group $\Ad_{\h}(S) \subset \GL(\h)$ is unipotent. Thus, it is algebraic and do not contain any cocompact algebraic subgroup: this means that $S$ is discompact. Moreover, this subgroup is amenable (since it is solvable) and by compactness of $M$, there exists a finite measure $\mu$ which is preserved by the $S$-action on $M$. Therefore, Theorem \ref{thm:embed} ensures that there exists Zimmer points for $S$. Let $x$ be one of them.

Proposition \ref{prop:morphisme_surjectif} then shows that there exists an algebraic subgroup $P^{\hx} < P$ and an algebraic surjective morphism $\rho : P^{\hx} \rightarrow \Ad_{\h}(S)$. Our work consists in proving that such a morphism does not exist.

Since $P^{\hx}$ is algebraic, it admits an algebraic Levi decomposition $P^{\hx} = (LT) \ltimes U$, where $L$, $T$ and $U$ are algebraic subgroups of $P^{\hx}$, with $L$ real semi-simple, $U$ unipotent and $T$ a torus (see \cite{ratner}, Theorem 4.4.7). The group $\Ad_{\h}(S)$ being itself a linear unipotent group, the semi-simple subgroup $\rho(L) < \Ad_{\h}(S)$ has to be trivial. Moreover, $\rho$ being an algebraic morphism, it sends elliptic (resp. hyperbolic) elements of $P^{\hx}$ on elliptic (resp. hyperbolic) elements of $\Ad_{\h}(S)$ (see \cite{ratner}, Corollary 4.3.6). Thus, we must have $\rho(T) =\{\id\}$. Finally, the restriction $\rho|_U : U \rightarrow \Ad_{\h}(S)$ must be onto: it is not restrictive to assume that $P^{\hx}$ is unipotent.

\subsubsection{Unipotent subgroups of $P$}

Choose coordinates $x_1,\ldots,x_{n+2}$ on $\R^{2,n}$ in such a way that the quadratic form is written $2x_1x_{n+2} + 2x_2x_{n+1} + x_3^2 + \cdots x_n^2$ and such that the parabolic subgroup $P < \PO(2,n)$ is the stabilizer of the isotropic line $[1:0:\cdots:0]$. In such a basis, any matrix of $\o (2,n) $ has the form

\begin{equation*}
\begin{pmatrix}
a&b & u_1 & \cdots & u_{n-2} & \alpha & 0 \\
c &d & v_1 & \cdots & v_{n-2} & 0     & -\alpha \\
-z_1 & -w_1 &     &        &     &  -v_1  & -u_1 \\
\vdots & \vdots &     &   A    &     & \vdots & \vdots \\
-z_{n-2} & -w_{n-2} &   &        &     &  -v_{n-2}  & -u_{n-2} \\
\beta & 0 & w_1 & \cdots & w_{n-2} &  -d    & -b \\
0 & -\beta & z_1 & \cdots & z_{n-2} &  -c    & -a
\end{pmatrix}
\end{equation*}
where all the letters denote real numbers, except $A$ which is an element of $\o(n-2)$. Moreover the Lie algebra of $P$ corresponds to matrices verifying $c=z_1=\cdots=z_{n-2}=\beta=0$.

In such coordinates, let $\u_{\text{max}}$ be the unipotent subalgebra of $\p$ composed with the strictly upper-triangular matrices
\begin{equation*}
\begin{pmatrix}
0&t & u_1 & \cdots & u_{n-2} & \alpha & 0 \\
 &0 & v_1 & \cdots & v_{n-2} & 0      & -\alpha \\
 &  &     &        &     &  -v_1  & -u_1 \\
 &  &     &    0   &     & \vdots & \vdots \\
 &  &     &        &     &  -v_{n-2}  & -u_{n-2} \\
 &  &     &        &     &  0     & -t \\
 &  &     &        &     &        & 0
\end{pmatrix}
\in \p.
\end{equation*}

\begin{fact*}
This Lie algebra $\u_{\max}$ is isomorphic to a semi-direct product $\R \ltimes \heis^{\C}(2n-3)$.
\end{fact*}

\begin{proof}
Recall that $\heis^{\C}(2n-3)$ is $\C^{n-2} \oplus \R. Z$ endowed with a nilpotent Lie algebra bracket such that $\R.Z$ is the center. A basic linear algebra computation shows that $\u_{\max} \cap \{t=0\}$ is isomorphic to $\heis^{\C}(2n-3)$, the correspondence being given by
\begin{equation*}
M \in \u_{\max} \cap \{t=0\} \mapsto (u_1+iv_1,\ldots ,u_{n-2}+iv_{n-2}) + \alpha Z.
\end{equation*}
Moreover, another elementary computation gives that $\u_{\max} \cap \{t=0\}$ is an ideal of $\u_{max}$ of codimension $1$ and that $\u_{\max} \simeq \R \ltimes \heis^{\C}(2n-3)$ where $\R$ acts by derivations on $\heis^{\C}(2n-3)$ by
\begin{equation*}
\forall t \in \R, u,v \in \R^{n-2}, \ \alpha \in \R, \ \ t . ((u+iv) + \alpha.Z) = (tv + i.0) + 0.Z \in \C^d \oplus \R.Z.
\end{equation*}
\end{proof}

Note $U_{\max} = \exp(\u_{\max})$. We now prove the

\begin{prop}
Every connected unipotent subgroup of $P$ is conjugated (in $P$) to a subgroup of $U_{\max}$.
\end{prop}

\begin{proof}
Let $U$ be a unipotent subgroup of $P$. In the same coordinates $x_1,\ldots,x_{n+2}$, the group $P$ is given by
\begin{equation*}
P = 
\left \{
\begin{pmatrix}
a & *&* \\
  & A &* \\
  & & a^{-1}
\end{pmatrix}
\in O(2,n)
\ ; \ a \in \R^*, \ A \in O(1,n-1)
\right \} 
\text{ mod } \{\pm \id\}
\end{equation*}
In what follows, everything is implicitly considered modulo $\{\pm \id\}$. Let $\pi_{\ell} : P \rightarrow \R^* \times O(1,n-1)$ be the map associating the diagonal matrix
$
\begin{pmatrix}
a & & \\
  & A & \\
  & & a^{-1}
\end{pmatrix}
$
. It is a Lie group morphism, and since $U$ is unipotent, we have $\pi_{\ell}(U) = \{1\} \times U'$ where $U'$ is now a unipotent subgroup of $O(1,n-1)$.

\begin{lem}
Let $U'$ be a connected unipotent subgroup of $O(1,n-1)$. There exists a non-zero isotropic vector of $\R^{1,n-1}$ which is fixed by all the elements of $U'$.
\end{lem}

\begin{proof}
Assume $\u' \neq 0$ and let $\z$ be the center of $\u'$. By hypothesis, in a suitable basis all elements of $\u'$ are upper triangular with $0$'s on the diagonal. The subspace $E = \bigcap_{Z \in \z} \ker Z \subset \R^{1,n-1}$ is thus non-trivial. Every $X \in \u'$ leaves $E$ stable, and necessarily its orthogonal relatively to the Lorentzian scalar product of $\R^{1,n-1}$. When we restrict the matrices of $\u'$ to $E^{\perp}$, we obtain a subalgebra of $\gl(E^{\perp})$ which is composed with nilpotent linear morphisms. Then, Engel's Theorem (\cite{knapp}, Theorem 1.35) gives us a line in $E^{\perp}$ which belongs to the kernel of every element of $\u'$. Therefore, $E \cap E^{\perp} \neq 0$ and contains a vector $v \neq 0$ such that $X(v) = 0$ for all $X \in \u'$.
\end{proof}

Consequently, the unipotent subgroup $U' < O(1,n-1)$ is contained in the stabilizer of some isotropic vector of $\R^{1,n-1}$. Such a stabilizer is conjugated, in $O(1,n-1)$, to
\begin{equation*}
\left \{
\begin{pmatrix}
1 & * & * \\
  & A & * \\
  & & 1
\end{pmatrix}
\in O(1,n-1)
 \ ; \ A \in O(n-2)
\right \}.
\end{equation*}
The compact group $O(n-2)$ does not contain unipotent elements, and finally there is $g \in O(1,n-1)$ such that 
\begin{equation*}
gU'g^{-1} \subset 
\left \{
\begin{pmatrix}
1 & * & * \\
  & \id & * \\
  & & 1
\end{pmatrix}
\in O(1,n-1)
\right \}.
\end{equation*}
Now, if we note $p := 
\begin{pmatrix}
1 & & \\
  & g & \\
  & & 1
\end{pmatrix}
\in P
$
, we see that $\Ad(p) \u \subset \u_{\max}$.
\end{proof}

Finally, the existence of a Zimmer point for $S$ gives a Lie subalgebra $\u < \u_{\max}$ and a Lie algebra morphism 
\begin{equation}
\u \rightarrow \ad_{\h}(\heis^{\H}(7)) \simeq \heis^{\H}(7)
\end{equation}
which is onto. We are now in position to prove that such a morphism does not exist, contradicting the existence of a conformal action of $H$ on $(M,[g])$. We thus finish this section with the following algebraic lemma.

\subsubsection{Algebraic incompatibility}

\begin{lem}
If $\u \subset \u_{\max} \simeq \R \ltimes \heis^{\C}(2n-3)$ is a subalgebra, a Lie algebra morphism $\u \rightarrow \heis^{\H}(7)$ is never onto.
\end{lem}

\begin{proof}
Let $Z_i$, $Z_j$ et $Z_k$ be three linearly independent elements of the center of $\heis^{\H}(7)$ and $U$, $U_i$, $U_j$, $U_k$ elements of $\heis^{\H}(7)$ whose bracket relations are
\begin{center}	
\begin{tabular}{|c|c|c|c|c|}
\hline 
[.,.] & $U$ & $U_i$ & $U_j$ & $U_k$ \\
\hline
$U$ & 0 & $Z_i$ & $Z_j$ & $Z_k$\\
\hline
$U_i$ & $-Z_i$ & 0 & $Z_k$ & $-Z_j$ \\
\hline
$U_j$ & $-Z_j$ & $-Z_k$ & 0 & $Z_i$\\
\hline
$U_k$ & $-Z_k$ & $Z_j$ & $-Z_i$ & 0 \\
\hline
\end{tabular}.
\end{center}
Let $f : \u \rightarrow \heis^{\H}(7)$ be a surjective morphism. We choose $X_1,X_2,X_3,X_4$ pre-images in $\u$ of $U, \, U_i, \, U_j, \ U_k$ respectively. Since $\u_{\max} \simeq \R \ltimes \heis^{\C}(2n-3)$, each $X_i$ can be written $X_i = (t_i,X_i')$, where $t_i \in \R$ and $X_i' \in \heis(2n-3)$. We claim that $t_1=t_2=t_3=t_4=0$.

\begin{proof}
Necessarily, $t_2X_1-t_1X_2$, $t_3X_1-t_1X_3$, $t_4X_1-t_1X_4 \: \in \u \cap \heis^{\C}(2n-3)$. We compute that
\begin{align*}
 f([t_2X_1-t_1X_2,t_3X_1-t_1X_3]) & = -t_1t_2Z_j + t_1t_3Z_i + t_1^2Z_k \\
 f([t_2X_1-t_1X_2,t_4X_1-t_1X_4]) & = -t_1t_2Z_k + t_1t_4Z_i -t_1^2Z_j \\
 f([t_3X_1-t_1X_3,t_4X_1-t_1X_4]) & = -t_1t_3Z_k + t_1t_4Z_j + t_1^2Z_i .
\end{align*}  
These elements belong to $D := f(\u \cap [\heis^{\C}(2n-3),\heis^{\C}(2n-3)]) \subset \heis^{\H}(7)$, but since $[\heis^{\C}(2n-3),\heis^{\C}(2n-3)]$ is $1$-dimensional, we have $\dim D \leq 1$.
Therefore,
\begin{equation*}
 \Rk
 \begin{pmatrix}
  t_1t_3 & -t_1t_2 & t_1^2 \\
  t_1t_4 & -t_1^2 & -t_1t_2 \\
  t_1^2 & t_1t_4 & -t_1t_3
 \end{pmatrix}
\leq 1.
\end{equation*}
In particular, the minor
$\begin{vmatrix}
 t_1t_4 & -t_1^2 \\
 t_1^2 & t_1t_4
\end{vmatrix}$
must vanish. Then, $t_1 = 0$. In the same way, we prove $t_2 = t_3 = t_4 = 0$.
\end{proof}

We still note $D = f(\u \cap [\heis^{\C}(2n-3),\heis^{\C}(2n-3)])$. Now, we know that $X_1,X_2,X_3,X_4 \in \heis^{\C}(2n-3)$. Consequently, we must have
\begin{align*}
f([X_1,X_2]) = [U,U_i] = Z_i \in D \\
f([X_1,X_3]) = [U,U_j] = Z_j \in D.
\end{align*}
This contradicts $\dim D \leq 1$, and the existence of $f$.
\end{proof}

\subsection{Orbits of Zimmer points}
\label{ss:orbits}

As announced, we now consider conformal actions of Lie groups locally isomorphic to $\SO(1,k)$ or $\SU(1,k)$. Such actions exist: we can embed $\PO(1,k)$ into $\PO(2,k) = \Conf(\Ein^{1,k-1})$ and $\PSU(1,k)$ into $\PO(2,2k) = \Conf(\Ein^{1,2k-1})$, therefore these groups act on compact Lorentz manifolds of dimension $k$ and $2k$ respectively. Studying orbits of Zimmer points in the general situation, we will prove that it is not possible to find an action of these groups on compact Lorentz manifolds with smaller dimension.

Before going on with the proof of Theorem \ref{thm:main}, we take a short detour and describe general properties of Zimmer points of a conformal action.

\subsubsection{General geometric informations on Zimmer points}
\label{sss:general_zimmer_conformal}

Let $(M^n,[g])$, $n \geq 3$, be a manifold endowed with a conformal class of metrics of signature $(p,q)$. We note $G=\PO(p+1,q+1)$ and $P<G$ the stabilizer of a fixed point $x_0 \in \Ein^{p,q}$. Let $(M,\mathcal{C})$ be the corresponding normalized Cartan geometry modeled on $\Ein^{p,q} = G/P$, with Cartan bundle $\pi : \hat{M} \rightarrow M$ and Cartan connection $\omega \in \Omega^1(\hat{M},\g)$. Recall that a Lie group $H$ acts conformally on $(M,[g])$ if and only if it acts by automorphisms of $(M,\mathcal{C})$. The Lie algebra $\h$ is seen as a Lie subalgebra $\h \subset \Kill(M,\mathcal{C})$ of Killing vector fields. We still note $\iota : \hat{M} \rightarrow \Mon(\h,\g)$ the map defined in (\ref{align:iota}), Section \ref{ss:embed}. If $x \in M$, we note $\h_x = \{X \in \h \ | \ X_x = 0\}$ the Lie algebra of the stabilizer of $x$ and $\mathcal{O}_x := H.x$ the $H$-orbit of $x$.

\begin{lem}
\label{lem:zimmer_stabilizer}
If $x \in M$ is a Zimmer point for some subgroup $S<H$, then $\Ad_{\h}(S) \h_x = \h_x$.
\end{lem}

\begin{rem}
This lemma is valid for every Cartan geometry.
\end{rem}

\begin{proof}
Since an automorphism $f \in \Aut(M,\mathcal{C})$ is covered by a bundle morphism $\hat{f}$, it fixes a point $x \in M$ if and only if $\hat{f}$ preserves the fiber $\pi^{-1}(x)$ over $x$. In particular, a Killing vector field $X$ vanishes at a point $x$ if and only if its lift $\hat{X}$ is tangent to the fiber $\pi^{-1}(x)$, or equivalently if $\forall \hx \in \pi^{-1}(x)$, $\omega_{\hx}(\hat{X}_{\hx}) \in \p$. Therefore, any $X \in \h$ vanishes at some point $x$ if and only if for some (equivalently for all) $\hx \in \pi^{-1}(x)$, we have $\iota(\hx)(X) \in \p$. Since $x$ is a Zimmer point for $S$, for all $s \in S$, there exists $p \in P$ such that $\forall X \in \h$, $\iota(\hx)(\Ad(s)X) = \Ad(p) \iota(\hx) (X)$. Therefore, $X \in \h_x \Rightarrow \iota(\hx)(X) \in \p \Rightarrow \Ad(p) \iota(\hx) (X) = \iota(\hx)(\Ad(s)X) \in \p \Rightarrow \Ad(s)X \in \h_x$.
\end{proof}

In what follows, we identify $\h / \h_x \simeq T_x \mathcal{O}_x \subset T_xM$. Thus, $\h / \h_x$ inherits a conformal class $[q_x]$ of quadratic forms obtained by restricting the conformal class $[g_x]$ of $T_xM$. Since it preserves $\h_x$, the group $\Ad_{\h}(S)$ naturally acts on $\h / \h_x$, we simply note $\bar{\Ad}(S)$ this action on the quotient.

\begin{prop}
\label{prop:zimmer_conformal}
Let $n \geq 3$ and $(M^n,[g])$ be a compact manifold endowed with a conformal class $[g]$ of non-degenerate metrics. Let $H < \Conf(M,[g])$ be a Lie subgroup, $S < H$ a closed subgroup and $x$ a Zimmer point for $S$. Then, we have $\bar{\Ad}(S) \subset \Conf(\h / \h_x , [q_x])$.
\end{prop}

\begin{proof}
We explained in Section \ref{sss:cartan_easy} that there exists a conformal class $[Q]$ on $\g / \p$ of signature $(p,q)$, invariant under the (quotiented) adjoint action $\bar{\Ad}(P)$, and a family of linear conformal identifications $\varphi_{\hx} : (T_x M,[g_x]) \rightarrow (\g / \p,[Q])$ that satisfy a natural equivariant relation.

As we noted before, for all $\hx \in \pi^{-1}(x)$, $\iota(\hx)$ sends $\h_x$ into $\p$. Therefore, it defines a map $\psi_{\hx} : \h / \h_x \rightarrow \g / \p$ such that the following diagram is commutative
\begin{equation*}
\xymatrix{
\llap{$\iota_{\hx} :$ } \h \ar[r] \ar[d] & \g \ar[d] \\
\llap{$\psi_{\hx} :$ }\h / \h_x \ar[r] & \g / \p
}
\end{equation*}
and it comes from the definitions that $\psi_{\hx}$ coincides with the restriction of $\varphi_{\hx}$ to $T_x \mathcal{O}_x \simeq \h / \h_x$. Thus, since $\varphi_{\hx}$ is conformal, $\psi_{\hx} : (\h/\h_x , [q_x]) \rightarrow (\g / \p, [Q])$ is a conformal linear injective map.

Fix $s \in S$. There exists $p \in P$ such that $\forall X \in \h$, $\iota_{\hx}(\Ad(s) X) = \Ad(p) \iota_{\hx}(X)$, and then $\psi_{\hx}(\bar{\Ad}(s)\bar{X}) = \bar{\Ad}(p) \psi_{\hx}(\bar{X})$, where the bars mean that we are in the quotient $\h / \h_x$ or $\g / \p$. We now compute
\begin{align*}
\forall \bar{X} \in \h / \h_x, \ q_x(\bar{\Ad}(s) \bar{X}) & = \lambda Q(\psi_{\hx}(\bar{\Ad}(s) \bar{X})) \text{ for some } \lambda > 0 \text{ since } \psi_{\hx} \text{ is conformal} \\
& = \lambda Q (\bar{\Ad}(p) \psi_{\hx}(\bar{X})) \\
& = \lambda \lambda' Q(\psi_{\hx}(\bar{X})) \text{ for some } \lambda' > 0 \text{ since } \bar{\Ad}(p) \in \Conf(\g/\p , [Q]) \\
& = \lambda' q_x(\bar{X}),
\end{align*}
proving that $\bar{\Ad}(s) \in \Conf(\h / \h_x, [q_x])$.
\end{proof}

We come back to the proof of Theorem \ref{thm:main} and until the end of this section, we assume that a connected Lie group $H$ locally isomorphic to $\SO(1,k)$ or $\SU(1,k)$, $k \geq 2$, acts conformally on a compact Lorentz manifold $(M^n,g)$, $n \geq 3$. 

\begin{dfn}
Let $V$ be a finite-dimensional vector space, endowed with a Lorentz quadratic form $q$. If $V' \subset V$ is a subspace, the restriction $q' := q|_{V'}$ is either non-degenerated, with Riemannian or Lorentz signature, or degenerated and positive, with a $1$-dimensional kernel. Such a quadratic form $q'$ will be said \textit{sub-Lorentzian}.
\end{dfn}

Remark that if $V$ is endowed with a sub-Lorentzian quadratic form $q$, then a subspace $V' \subset V$ totally isotropic with respect to $q$ is at most $1$-dimensional.

\subsubsection{Conformal actions of $\o(1,k)$}
\label{sss:o1k}

We treat here the case where $H$ is a connected Lie group locally isomorphic to $\SO(1,k)$, $k \geq 2$. We start collecting some algebraic materials on $\SO(1,k)$ that we will use.

\vspace*{0.2cm}

We have seen that the Lie algebra $\o(1,k)$ admits the root-space decomposition $\o(1,k) = \h_{-\alpha} \oplus \a \oplus \m \oplus \h_{\alpha}$, where $\m \simeq \o(k-1)$. Set $\h_0 := \a \oplus \m$ the centralizer of the Cartan space $\a$.

\begin{lem}~
\label{lem:trick_o1k}
\begin{enumerate}
\item The adjoint action $\ad(\h_0)$ on $\h_{\pm \alpha}$ is irreducible.
\item Let $\h' \subset \o(1,k)$ be a Lie subalgebra such that $[\h_{\alpha},\h'] \subset \h'$. If $\h' \cap \h_{-\alpha} \neq 0$, then $\h' = \o(1,k)$.
\end{enumerate}
\end{lem}

\begin{proof}
(1) An elementary computation using the linear representation of $\o(1,k)$ given in Section \ref{sss:preliminaries_orthogonal} shows that the adjoint action of $\h_0$ on $\h_{\pm \alpha}$ is conjugated to the standard linear action of $\R \oplus \o(k-1)$ on $\R^{k-1}$ ($\R$ acting by homotheties).

(2) We note $\theta : X \mapsto - ^{t} \! X$ a Cartan involution (with respect to our root-space decomposition). Let $Y \in \h' \cap \h_{-\alpha}$ non-zero and $X \in \h_{-\alpha}$. Since $\theta(X) \in \h_{\alpha}$, we have $[\theta(X),Y] \in \h' \cap \h_0$. An elementary matrix computation using the linear representations given in Section \ref{sss:preliminaries_orthogonal} shows $[[Y,\theta(X)],Y] = B_{\theta}(X,Y) Y - \frac{1}{2} B_{\theta}(Y,Y) X$, where $B_{\theta}(X,Y):= \Tr(X\theta(Y))$ (it is a negative definite quadratic form on $\h$ since $\theta$ is a Cartan involution). Therefore, since $Y \neq 0$ and $[[Y,\theta(X)],Y] \in \h'$, we get $X \in \h'$, proving $\h_{-\alpha} \subset \h'$. Using the matrix representation of $\o(1,k)$, we verify that $\h_0 = [\h_{-\alpha},\h_{\alpha}]$ and $\h_{\alpha} = [\h_0,\h_{\alpha}]$. We then have $\h_0 \subset \h'$, and $\h_{\alpha} \subset \h'$.
\end{proof}

Recall that we assume that $H$ acts faithfully and conformally on $(M,[g])$. Let $S < H$ be the connected Lie subgroup whose Lie algebra is $\s = \a \oplus \h_{\alpha}$. It will not be difficult to see that $S$ is discompact and amenable, ensuring the existence of a Zimmer point for $S$. We first establish the following proposition on the possible orbits of such a point.

\begin{prop}
\label{prop:orbites_o1k}
Let $x \in M$ be a Zimmer point for $S$ and $H_x < H$ be the stabilizer of $x$. If $k \geq 4$, then the Lie algebra $\h_x$ is one of the following Lie subagebras of $\h$:
\begin{enumerate}
\item $\m \oplus \h_{\alpha}$ ;
\item $\a \oplus \m \oplus \h_{\alpha}$ ;
\item $\h$.
\end{enumerate}
When $\h_x = \a \oplus \m \oplus \h_{\alpha}$, the orbit $H.x$ is conformal to the standard Riemannian round sphere $\S^{k-1}$.
\end{prop}

\begin{proof}
We keep the same notations for the root-space decomposition. Let $A \in \a$ such that $\alpha(A) = 1$ and $u_t := \bar{\Ad}(\e^{tA}) \in \GL(\h / \h_x)$. If $[q_x]$ is the conformal class of subLorentzian metrics on $\h / \h_x$ that we considered in Proposition \ref{prop:zimmer_conformal}, we have $u_t \in \Conf(\h / \h_x , [q_x])$.

Assume $\h_x \neq \h$. Since $x$ is a Zimmer point for $S$, we have $[\s,\h_x] \subset \h_x$ (Lemma \ref{lem:zimmer_stabilizer}) and with Lemma \ref{lem:trick_o1k}, this implies $\h_x \cap \h_{-\alpha} = 0$. If we note $\pi_x : \h \rightarrow \h / \h_x$ the natural projection, $\dim \pi_x (\h_{-\alpha}) = k-1 \geq 3$ and the restriction of $q_x$ to $\pi_x(\h_{-\alpha})$ cannot vanish identically. Since $u_t$ coincides with $\e^{-t}\id$ on this latter space and is conformal for $q_x$, we must have $u_t^* q_x = \e^{-2t}q_x$.

Therefore, $q_x$ must vanish identically on the subspace $\pi_x(\h_0 \oplus \h_{\alpha}) \subset \h / \h_x$. Since $q_x$ is subLorentzian, this means that $\dim \pi_x(\h_0 \oplus \h_{\alpha}) \leq 1$. We then distinguish two cases.
\begin{itemize}
\item The first one is $\h_0 \subset \h_x$. Since $\pi_x(\h_{\alpha})$ is at most $1$-dimensional, $\dim \h_x \cap \h_{\alpha} \geq k-2 > 0$. The adjoint action of $\ad(\h_0)$ on $\h_{\alpha}$ being irreducible (Lemma \ref{lem:trick_o1k}), we then have $[\h_0,\h_x \cap \h_{\alpha}] = \h_{\alpha} \subset \h_x$. Finally, $\h_x = \h_0 \oplus \h_{\alpha}$ and we are in situation (2).
\item The second is when $\h_0 \nsubset \h_x$. We then have $\h_x = \h_{0,x} \oplus \h_{\alpha}$ where $\h_{0,x}$ is a codimension $1$ subalgebra of $\h_0 = \a \oplus \m$. Since $k \geq 4$, $\m \simeq \o(k-1)$ does not admit a codimension $1$ subalgebra, proving $\h_{0,x} = \m$ and we are in situation (1).
\end{itemize}
Now, assume that $\h_x = \a \oplus \m \oplus \h_{\alpha}$. Since $k > 2$, there are only two possibilities for $H$: it is either isomorphic to $\PSO(1,k)$ or $\Spin(1,k)$. In fact, $\Spin(1,k)$ is the universal cover of $\PSO(1,k)$, and the covering $\Spin(1,k) \rightarrow \PSO(1,k)$ is $2$-sheeted.

If $H =\PSO(1,k)$, it naturally acts on $\S^{k-1} \simeq \Ein^{0,k-1}$ (see Section \ref{sss:einstein_cartan}) and the stabilizer of $x_0 := [1:0:\cdots : 0]$ is the only closed subgroup of $\PSO(1,k)$ whose Lie algebra is $\h_x$. The homogeneous space $H / H_x$ is thus diffeomorphic to $\S^{k-1}$.

If $H = \Spin(1,k)$, the projection of the covering $p : \Spin(1,k) \rightarrow \PSO(1,k)$ sends $H_x$ to the stabilizer of $x_0 \in \S^{k-1}$ in $\PSO(1,k)$. Therefore, $p$ induces a local diffeomorphism $H/H_x \rightarrow \S^{k-1}$. This must be a covering, and since $k > 2$, a diffeomorphism.

Thus, in both cases the orbit $\mathcal{O}_x := H.x$ is a properly embedded submanifold of $M$, diffeomorphic to $\S^{k-1}$. At last, it is a standard property of homogeneous spaces that the isotropy representation $\rho_x : H_x \rightarrow \GL(T_x \mathcal{O}_x)$ is conjugated to the representation $\bar{\Ad} : H_x \rightarrow \GL(\h/\h_x)$ induced by the adjoint representation. Since $\rho_x(H_x)$ is conformal, we have in fact $\bar{\Ad} : H_x \rightarrow \Conf(\h / \h_x, [q_x])$. If we note $H'$ the closed subgroup of $H_x$ corresponding to $\m \simeq \o(k-1)$ ($H' = \PSO(k-1)$ or $\Spin(k-1)$), by compactness we have $\bar{\Ad}(H') \subset \Isom(\h/\h_x,q_x)$. By Lemma \ref{lem:trick_o1k}, the action of $H'$ on $\h / \h_x$ is conjugated to the linear action of $\SO(k-1)$ on $\R^{k-1}$. Since this action leaves $q_x$ invariant, $q_x$ must be Euclidian. Therefore, the Lorentz conformal class $[g]$ induces a conformal Riemannian structure on  $\mathcal{O}_x \simeq_{\text{diff}} \S^{k-1}$, which must be invariant under the standard action of $\PSO(1,k)$, and has to be the standard Riemannian structure on $\S^{k-1}$.
\end{proof}

\begin{cor}
\label{cor:majoration_o1k}
When a Lie group locally isomorphic to $\SO(1,k)$, $k \geq 2$, acts conformally on a compact Lorentz manifold $(M,g)$, with $\dim M \geq 3$, we have $k \leq \dim M$.
\end{cor}

\begin{proof}
In order to ensure that there exists Zimmer points for $S$, we prove the
\begin{lem}
\label{lem:discompact_o1k}
The Lie subgroup $S$ is discompact.
\end{lem}

\begin{proof}
We compute that with respect to the root-space decomposition $\h = \h_{-\alpha} \oplus \a \oplus \m \oplus \h_{\alpha}$, we have
\begin{equation*}
\bar{\Ad_{\h}(S)}^{\text{Zar}} =
\left \{
\begin{pmatrix}
x^{-1} \id & & \\
 & \id & \\
 & & x \id
\end{pmatrix}
\exp(\ad_{\h} (X)) \ ; \ x \in \R^*, \ X \in \h_{\alpha}
\right \}
= \R^* \ltimes \exp(\ad_{\h}(\h_{\alpha})).
\end{equation*}
where $\R^*$ acts on $\exp(\ad_{\h}(\h_{\alpha}))$ by $x.\exp(\ad(X)) = \exp(x \ad(X))$. Indeed, the linear group $\exp(\ad(\h_{\alpha})) \subset \GL(\h)$ is unipotent, so it is algebraic and if $A \in \a$ is such that $\alpha(A)=1$, the group $\{\exp(t\ad(A)), \ t \in \R\}$ normalizes $\exp(\ad(\h_{\alpha}))$ and acts on it via powers of $\e^{t}$. 

Thus, an algebraic cocompact $S'$ subgroup of $\bar{\Ad_{\h}(S)}^{\text{Zar}}$ must contain the diagonal factor $\R^*$, and since a unipotent linear group does not admit any proper algebraic cocompact subgroup, we must have $S' = \R^* \ltimes \exp(\ad_{\h}(\h_{\alpha}))$.
\end{proof}

Since $S$ is discompact and amenable, it acts on $M$ preserving some non-trivial finite measure and Theorem \ref{thm:embed} gives us the existence of a Zimmer point $x$ for $S$. Proposition \ref{prop:orbites_o1k} describes three possibilities for the orbit $\mathcal{O}_x := H.x$.
\begin{itemize}
\item In the first case, we have $\dim \mathcal{O}_x = k$. In particular, we have $k \leq n$.

\item In the second case, the orbit $\mathcal{O}_x$ being a Riemannian submanifold, we must have $\dim \mathcal{O}_x \leq \dim M - 1$.  Since $\mathcal{O}_x$ has dimension $k-1$, we get $k \leq n$.

\item When $\mathcal{O}_x = \{x\}$, the isotropy representation $\rho_x : H \rightarrow \Conf(T_xM,[g_x])$ gives rise to a Lie algebra morphism $\iota : \o(1,k) \rightarrow \co(1,n-1)$. Moreover, $\iota(\o(1,k)) = [\iota(\o(1,k)),\iota(\o(1,k))] \subset [\co(1,n-1),\co(1,n-1)] = \o(1,n-1)$. Since $\o(1,k)$ is simple, this morphism is trivial or injective. But if $\iota = 0$, $\rho_x$ would be trivial by connectedness of $H$. In this situation, Thurston's stability theorem says that either $H$ acts trivially near $x$, or $H^1(H,\R) \neq 0$ (see \cite{cairns_ghys}, p. 140). The first case contradicts the faithfulness of the action, and the second the simplicity of $H$. Therefore, $\iota$ is an embedding of $\o(1,k)$ into $\o(1,n-1)$, and $k \leq n-1$.
\end{itemize}
In every case, we get $k \leq \dim M$.
\end{proof}

\subsubsection{Conformal actions of $\su(1,k)$}
\label{sss:su1k}

We finally treat the case where $H$ is a connected Lie group locally isomorphic to $\SU(1,k)$, $k \geq 2$. We state some algebraic facts that we will use. Recall that $\su(1,k)$ admits the root-space decomposition $\su(1,k) = \h_{-2\alpha} \oplus \h_{-\alpha} \oplus \a \oplus \m \oplus \h_{\alpha} \oplus \h_{2\alpha}$, where $\m \simeq \u(k-1)$. We note $\h_0 := \a \oplus \m$ the centralizer of the Cartan subspace $\a$.

\begin{lem}
\label{lem:trick_su1k}
Let $\h' \subset \su(1,k)$ be a Lie subalgebra such that $[\h_{\alpha} , \h'] \subset \h'$. If $\h' \cap \h_{-\alpha} \neq 0$, then $\h' = \su(1,k)$.
\end{lem}

\begin{proof}
We note $\theta : X \mapsto - ^{t} \! \bar{X}$ a Cartan involution (with respect to our root-space decomposition). Let $Y \in \h'\cap \h_{-\alpha}$ non-zero and $X \in \h_{-\alpha}$. Since $\theta(X) \in \h_{\alpha}$, we have $[\theta(X),Y] \in \h'$. An elementary matrix computation (using the matrices of Section \ref{sss:heisenberg}, proof of Lemma \ref{lem:heis_su_sp}) gives $[[Y,\theta(X)],Y] = B_{\theta}(X,Y) Y - \frac{1}{2} B_{\theta}(Y,Y) X$, where $B_{\theta}(X,Y) := \Tr(X\theta(Y))$ (it is a negative definite quadratic form on $\h$ since $\theta$ is a Cartan involution). Therefore, since $Y \neq 0$ and $[[Y,\theta(X)],Y] \in \h'$, we get $X \in \h'$, proving $\h_{-\alpha} \subset \h'$.

Using once more the matrix representation of $\su(1,k)$, we can compute that $[\h_{-\alpha},\h_{-\alpha}] = \h_{-2\alpha}$, $[\h_{\alpha},\h_{\alpha}] = \h_{2\alpha}$, $[\h_{-\alpha},\h_{\alpha}] = \h_0$ and $[\h_0,\h_{\alpha}] = \h_{\alpha}$. Since $\h_{-\alpha} \subset \h'$, we obtain successively $\h_{-2\alpha} \subset \h'$, $\h_0 \subset \h'$, $\h_{\alpha} \subset \h'$ and $\h_{2\alpha} \subset \h'$.
\end{proof}

Let $S < H$ be the connected Lie subgroup whose Lie algebra is $\s = \a \oplus \h_{\alpha} \oplus \h_{2\alpha}$.

\begin{lem}
\label{lem:discompact_su1k}
The Lie subgroup $S < H$ is discompact.
\end{lem}

\begin{proof}
The proof is strictly similar to the one of Lemma \ref{lem:discompact_o1k}, the only difference being that the adjoint action of the Cartan subgroup has two others eigenvalues.
\end{proof}

Since $S$ is amenable, it acts on $M$ preserving a finite measure and Theorem \ref{thm:embed} ensures the existence of Zimmer points for $S$. The aim of this section is to prove the following Proposition.

\begin{prop}
Let $x$ be a Zimmer point for $S$. Then, the Lie algebra of the stabilizer $H_x$ of $x$ in $H$ can be written $\h_x = \h_{0,x} \oplus \h_{\alpha} \oplus \h_{2\alpha}$, where $\h_{0,x}$ is a codimension $1$ subalgebra of $\h_0$. In particular, the orbit $H.x$ has dimension $2k$. Moreover, this orbit has Lorentz signature.
\end{prop}

As an immediate corollary, we get that if a Lie group locally isomorphic to $\SU(1,k)$ acts faithfully and conformally on a compact Lorentz manifold $(M,g)$ of dimension at least $3$, then $2k \leq \dim M$, finishing the proof of Theorem \ref{thm:main} in the case of real rank $1$ simple Lie groups.

\begin{proof}

\begin{fact}
\label{fact:one}
We have $\h_x \neq \h$.
\end{fact}

\begin{proof}
If not, we would have a Lie algebra embedding $\su(1,k) \hookrightarrow \o(1,n-1)$ (same arguments than in the proof of Corollary \ref{cor:majoration_o1k}). But the root-system of $\su(1,k)$ being $\{\pm \alpha, \pm 2\alpha \}$, it cannot be embedded into any $\o(1,N)$, $N \geq 2$, whose root-system is $\{\pm \alpha\}$.
\end{proof}

We reuse the notations of the proof of Proposition \ref{prop:orbites_o1k}. Let $\pi_x : \h \rightarrow \h / \h_x$ be the natural projection and let $A \in \a \subset \s$ such that $\alpha(A) = 1$. By Proposition \ref{prop:zimmer_conformal}, we know that $\ad(\a) \h_x$ and that $u_t := \bar{\Ad}(\e^{tX}) \in \Conf(\h / \h_x, [q_x])$. Thus, there exists $\lambda \in \R$ such that $u_t^* q_x = \e^{\lambda t} q_x$. If $v \in \h_{-\alpha}$, we have $q_x(u_t(\pi_x(v))) = \e^{\lambda t} q_x(\pi_x(v)) = \e^{-2t} q_x(\pi_x(v))$. But Lemma \ref{lem:trick_su1k} and Fact \ref{fact:one} imply that $\dim \pi_x(\h_{-\alpha}) = \dim \h_{-\alpha} \geq 2$. Since $q_x$ is sub-Lorentzian, it cannot vanish identically on $\pi_x(\h_{-\alpha})$, proving $\lambda = -2$.

Consequently, for every $\beta,\beta' \in \{0, \pm \alpha, \pm 2\alpha \}$, $\pi_x(\h_{\beta})$ and $\pi_x(\h_{\beta'})$ are orthogonal with respect to $q_x$ as soon as $\beta + \beta' \neq -2\alpha$. Indeed, noting $B_x$ the associated bilinear form, if $X \in \h_{\beta}$ and $Y \in \h_{\beta'}$, we have $B_x(u_t(\pi_x(X)),u_t(\pi_x(Y))) = \e^{t(\beta + \beta')(A)} B_x(\pi_x(X),\pi_x(Y)) = \e^{-2t}B_x(\pi_x(X),\pi_x(Y))$. Therefore, if $(\beta + \beta')(A) \neq -2$ we must have $B_x(\pi_x(X),\pi_x(Y))=0$. In particular, $\pi_x(\h_0 \oplus \h_{\alpha} \oplus \h_{2\alpha})$ is totally isotropic with respect to $q_x$ and must have dimension at most $1$, \textit{ie} $\h_x \cap (\h_0 \oplus \h_{\alpha} \oplus \h_{2\alpha})$ has codimension $\leq 1$ in $\h_0 \oplus \h_{\alpha} \oplus \h_{2\alpha}$.

On the other hand, by Lemma \ref{lem:trick_su1k}, $\h_x \cap \h_{-\alpha} = 0$. This implies $\h_x \cap \h_{-2\alpha} = 0$. Indeed, $\h_{-2\alpha}$ being $1$-dimensional, the contrary would be $\h_{-2\alpha} \subset \h_x$ and since $[\h_{-2\alpha},\h_{\alpha}] = \h_{-\alpha}$ we would have $\h_{-\alpha} \subset \h_x$. Since $\Ad(\e^{tA}) \h_x \subset \h_x$ and since this flow acts diagonally with different exponential rates on the $\h_{\beta}$'s, we obtain that $\h_x \subset \h_0 \oplus \h_{\alpha} \oplus \h_{2\alpha}$.

\begin{fact}
We have $\h_0 \nsubset \h_x$ (implying that $\h_0 \cap \h_x$ has codimension $1$ in $\h_0$).
\end{fact}

\begin{proof}
Assume $\h_0 \subset \h_x$. Since $\dim \h_{\alpha} \geq 2$ we must have $\h_x \cap \h_{\alpha} \neq 0$ (if not $\dim \pi_x(\h_{\alpha}) \geq 2$). The adjoint action $\ad(\h_0)$ on $\h_{\alpha}$ being irreducible, we get $[\h_0,\h_{\alpha} \cap \h_x] = \h_{\alpha} \subset \h_x$, and also $\h_{2\alpha} = [\h_{\alpha},\h_{\alpha}] \subset \h_x$. We then have $\h_x = \h_0 \oplus \h_{\alpha} \oplus \h_{2\alpha}$ and $\pi_x : \h_{-2\alpha} \oplus \h_{-\alpha} \rightarrow \h/ \h_x$ is a linear isomorphism. Since $u_t^* q_x = \e^{-2t}q_x$, $\pi_x(\h_{-2\alpha})$ is isotropic and orthogonal to $\pi_x(\h_{-\alpha})$. This means that $q_x$ is degenerate and $\Ker q_x = \pi_x(\h_{-2\alpha})$. But $\bar{\Ad}(S)$ acts conformally on $(\h/\h_x , q_x)$, and has to preserve $\Ker q_x$. This is a contradiction since $\h_{\alpha} \subset \s$ and $[\h_{-2\alpha} , \h_{\alpha}] = \h_{-\alpha}$, implying that we have some $X \in \h_{\alpha}$ such that $\bar{\Ad}(\e^{tX})$ does not preserve $\pi_x(\h_{-2\alpha})$ in $\h / \h_x$.
\end{proof}

Finally, the Lie algebra of the stabilizer of $x$ is
\begin{equation*}
\h_x = \h_{x,0} \oplus \h_{\alpha} \oplus \h_{2\alpha},
\end{equation*}
where $\h_{x,0} \subset \h_0$ is a subalgebra of codimension $1$. Moreover, in $(\h / \h_x, q_x)$, the lines $\pi_x(\h_{-2\alpha})$ and $\pi_x(\h_0)$ are isotropic and orthogonal to $\pi_x(\h_{-\alpha})$. This implies that $q_x$ is a Lorentz quadratic form, and $\dim \h / \h_x = \dim \h_{-2\alpha} + \dim \h_{-\alpha} + 1 = 2k$.
\end{proof}

\subsection{Semi-simple Lie groups of rank $2$}
\label{ss:rank2}

Let $H$ be a real rank $2$ semi-simple Lie group without compact factor. Assume that $H$ acts conformally on a compact Lorentz manifold $(M^n,g)$.

The result of Bader and Nevo cited in the introduction (\cite{badernevo}) shows that when $H$ is simple, it must be locally isomorphic to $\SO(2,k)$, with $3 \leq k \leq n$. Thus we are left to study the case where $H$ is not simple.

We conclude using a general result on compact parabolic geometries admitting large groups of automorphisms. It is stated in \cite{embed}, Theorem 1.5. It applies in our situation since the homogeneous model space of Lorentzian conformal geometry, namely $\Ein^{1,n-1} = \PO(2,n) / P$, is parabolic. Here we assume that a real rank $2$ semi-simple Lie group $H$ acts by conformal transformations on a compact Lorentz manifold. The theorem of \cite{embed} implies that this manifold is conformally diffeomorphic to some quotient $\Gamma \backslash \tilde{\Ein^{1,n-1}}$ where $\Gamma$ is a discrete subgroup of $\Conf(\tilde{\Ein^{1,n-1}}) \simeq \tilde{\PO(2,n)}$. In particular, $H$ can be locally embedded into $O(2,n)$ and we have an injective Lie algebra morphism $\h \hookrightarrow \o(2,n)$. The following lemma finishes the proof of Theorem \ref{thm:main}.

\begin{lem}
Let $\h$ be a semi-simple Lie algebra without compact factor, with $\R$-rank $2$ and non-simple. If $\h$ can be embedded into some $\o(2,N)$, $N \geq 3$, then $\h \simeq \o(1,k) \oplus \o(1,k')$, with $k,k' \geq 2$. Moreover, except when $\h = \o(1,2) \oplus \o(1,2) \simeq \o(2,2)$, we have $k+k' \leq N$.
\end{lem}

\begin{proof}
By hypothesis, $\h$ splits into $\h = \h_1 \oplus \h_2$, where $\h_1$ and $\h_2$ are simple, non-compact Lie algebra of $\R$-rank $1$. Therefore, $\h_1$ and $\h_2$ admit $A_1$ or $(BC)_1$ as restricted root systems. Since $\h$ has real rank $2$, its restricted root system can be realized into the restricted root system of $\o(2,N)$. The latter is $B_2$ and does not contain neither $A_1 \oplus (BC)_1$ nor $(BC)_1 \oplus (BC)_1$. This observation gives us directly $k,k' \geq 2$ such that $\h_1 \simeq \o(1,k)$ and $\h_2 \simeq \o(1,k')$.

Assume that $k \geq k'$ and $k > 2$. We want to prove $k+k' \leq N$. Considering the complexifications of the Lie algebras, it is enough to prove that if
\begin{equation*}
\rho : \o(n,\C) \oplus \o(m,\C) \hookrightarrow \o(N,\C),
\end{equation*}
with $n \geq m$, $n \geq 4$ and $N \geq 5$, then $n+m \leq N$. We treat this question by using results on linear representations of complex orthogonal Lie algebras. We note $\g_1 = \o(n,\C)$ and $\g_2 = \o(m,\C)$.

Since $\g_1 \oplus \g_2$ is semi-simple, $\rho$ is completely reducible. Let $\C^N = (\bigoplus_i V^i) \oplus E$ be a decomposition into irreducible subrepresentations, where for all $i$, $\rho|_{V^i}$ is non-trivial (implying $\dim V^i \geq 2$) and $E = \bigcap_{X \in \g_1 \oplus \g_2} \ker \rho(X)$. We distinguish two cases: either there exists $i$ such that $\rho|_{V^i}$ is faithful or for all $i$, the kernel of $\rho|_{V^i}$ is $\g_1$ or $\g_2$ (and $\rho|_{V^i}$ is in fact a faithful irreducible representation of $\g_2$ or $\g_1$ respectively).

In the first situation, let $V \subset \C^N$ be a faithful irreducible subrepresentation of $\rho$. By irreducibility, $\rho|_V$ can be written $\rho_1 \otimes \rho_2$ where $\rho_1 : \g_1 \rightarrow \gl(V_1)$ and $\rho_2 : \g_2 \rightarrow \gl(V_2)$ are faithful irreducible representations (see \cite{fulton}, p. 381). We claim that except when $n=m=3$, or $n=6$ and $m=3$, we always have $\dim V_1 \otimes V_2 \geq n+m$.

To see this, recall that if $p \geq 1$ and $p \neq 2$, the smallest dimension of a faithful irreducible representation of $\o(2p,\C)$ is $\min(2p,2^{p-1})$ and that if $p \geq 1$, the smallest dimension of a faithful irreducible representation of $\o(2p+1,\C)$ is $\min(2p+1,2^p)$). In both situations, the cases $2p$ or $2p+1$ correspond to the standard linear representation, and the cases $2^{p-1}$ or $2^p$ to the half-spin representation (even case) or spin representation (odd case), see \cite{fulton}, Propositions 19.22, 20.15, 20.20. Since $\o(4,\C) = \o(3,\C) \oplus \o(3,\C)$, a faithful irreducible representation $V$ of $\o(4,\C)$ is the tensor product of two irreducible representations of $\o(3,\C)$, hence $\dim V \geq 4$. Finally, if we note $d_n$ the smallest dimension of a faithful irreducible representation of $\o(n,\C)$, $n \geq 3$, we have
\begin{equation*}
d_n = 
\begin{cases}
2 \text{ if } n=3 \\
4 \text{ if } n \in \{4,5,6\} \\
n \text{ if } n \geq 7
\end{cases}
\end{equation*}
We see that $d_n d_m \geq n+m$ except when $n=m=3$ or $n=6$ and $m=3$. But in the latter situation, we prove directly that there does not exist an embedding $\o(6,\C) \oplus \o(3,\C) \hookrightarrow \o(8,\C)$. Indeed, these algebras have both rank $4$. Thus, we would deduce from this embedding that the root system of the first is included in the second, \textit{i.e.} that $A_3 \oplus A_1$ is included in $D_4$, but this is not true. This finishes the case where some irreducible subrepresentation $V$ is faithful.

Now assume that none of the $V^i$'s is faithful. Then, we can regroup the $V^i$'s together in such a way that $\C^N = V_1 \oplus V_2 \oplus E$ and $\rho = (\rho_1,\rho_2,0)$ where $\rho_1 : \g_1 \rightarrow \gl(V_1)$ and $\rho_2 : \g_2 \rightarrow \gl(V_2)$ are faithful representations. We claim that $\dim V_1 \geq n$ and $\dim V_2 \geq m$. Using the same notations as above, it is easy to observe that $d_k \geq k/2$ for all $k \geq 3$. Thus, we can assume that $V_1$ and $V_2$ are irreducible.

We use here that $\rho$ is an orthogonal representation of $\g_1 \oplus \g_2$ on $\C^N$. Let $Q$ be the non-degenerate quadratic form of $\C^N$ for which $\rho$ is skew-symmetric. We claim that $Q|_{V_i}$, $i \in \{1,2\}$, is non-degenerate, proving that $\rho_i$ is conjugated to a (faithful) representation of $\g_i$ in $\o(\dim V_i,\C)$ and finishing the proof.

Let $j \in \{1,2\}$ be the other index and $B$ the bilinear form associated to $Q$. Let $v \in V_i$ non-zero. The subspace $< \rho_i(X_1)\ldots \rho_i(X_k).v \ ; \ k \geq 1, \ X_1,\ldots,X_k \in \g_i > \subset V_i$ is a subrepresentation of $\rho_i$, and must be equal to $V_i$ by irreducibility. But any element of the form $\rho_i(X).v'$, $X \in \g_i$ and $v' \in V_i$, satisfies $\forall w \in V_j \oplus E, \ B(\rho_i(X).v',w) = -B(v',\rho_i(X)w) = 0$ since $\rho_i |_{V_j \oplus E} \equiv 0$. This proves that $V_i$ is orthogonal to $V_j \oplus E$ with respect to $Q$. Thus, if $Q|_{V_i}$ was degenerate, $Q$ would also be degenerate.
\end{proof}

\bibliographystyle{amsalpha}
\bibliography{references_article_semi_simples.bib}
\nocite{*}

\vspace*{6.5cm}

\hspace*{8cm}
\begin{minipage}[b]{0.55\linewidth}
Vincent \textsc{Pecastaing} \\
Laboratoire de Mathématique \\
Faculté des sciences d'Orsay \\
F-91405 Orsay Cedex \\
France \\
\\
\texttt{vincent.pecastaing@normalesup.org}
\end{minipage}

\end{document}